\documentclass[11pt,twoside]{amsart}    
\usepackage{tikz}
\usepackage{tikz-cd}
\usepackage[margin=1.2in]{geometry}
\usepackage{latexsym,amssymb,graphicx}
\usepackage{mathrsfs}
\usepackage{todonotes}
\usepackage{mathtools}
\usepackage[all]{xy}

\newtheorem{theorem}{Theorem}[section]
\newtheorem{lemma}[theorem]{Lemma}
\newtheorem{corollary}[theorem]{Corollary}

\newtheorem{proposition}[theorem]{Proposition}

\theoremstyle{definition}
\newtheorem{definition}[theorem]{Definition}
\newtheorem{example}[theorem]{Example}
\theoremstyle{remark}
\newtheorem{remark}[theorem]{Remark}

\numberwithin{equation}{section}

\newcommand{\mc}{\mathcal}

\newcommand{\spec}{\mathrm{Spec\ }}

\tikzset{
	allign/.style={anchor=north, rotate=90, inner sep=1mm}
}

\address{Onkar Kamlakar Kale, Department of Mathematics, Harish-Chandra Research Institute, Chhatnag Road, Jhunsi, Prayagraj (Allahabad) 211019, India.}
\email{onkars.kale@gmail.com}

\address{Girja S Tripathi, Department of Mathematics, IISER Tirupati, Karakambadi Road, Tirupati 517501, India.}
\email{girja@labs.iisertirupati.ac.in}

\title{constructible Witt theory of Schemes}

\author{Onkar Kamlakar Kale and Girja S Tripathi}

\begin{document}

\bibliographystyle{alpha}

\begin{abstract} 
We study the constructible Witt theory of \'etale sheaves of $\Lambda$-modules on a scheme $X$ for coefficient rings $\Lambda$ having finite characteristic not equal to 2 and prime to the residue characteristics of the scheme $X$. Our construction is based on the recent advances by Cisinski and D\'eglise on six-functor formalism for derived categories of \'etale motives and offers a background for the study of constructible Witt theory as a cohomological invariant for schemes. In the case of smooth complex algebraic varieties and finite coefficient rings, we show that the algebraic constructible Witt theory studied in this paper can be identified with the topological constructible Witt theory.
\end{abstract}
\maketitle

\section{Introduction}
Witt groups of sheaves of modules on topological spaces have been considered as a generalized cohomology theory by Woolf in \cite{woolf2008witt}
and by Woolf and Sch\"{u}rmann in \cite{schurmann2020witt}. These provide signature-type invariants for topological spaces taking values in Witt theories related to the coefficient ring and have descriptions as symmetric forms on intersection cohomology complexes.
This work grew out of our interest in the analogous algebraic setting particularly with an interest in algebraic counterpart for the interpretation of $L$-classes as stable homology operations from (topological) constructible Witt groups to the ordinary rational homology. 
We consider the derived categories of constructible \'etale sheaves of modules over suitable coefficient rings and use Balmer's theory of triangular Witt groups to define constructible Witt groups: Given a scheme $X$ and a ring $\Lambda$ having finite characteristic not equal to 2 and prime to residue characteristics of $X$, we define the constructible Witt groups $W^i_c(X_{\acute{e}t},\Lambda)$ as the Witt groups of triangulated category with duality $(D^b_{ctf} (X_{\acute{e}t},\Lambda), D_X(T))$ consisting of \'etale sheaves of $\Lambda$-modules having finite Tor-dimension and constructible cohomology sheaves. 

The identification in section \ref{constructible Witt theory of fields}, of constructible Witt theory $W^i_c((\spec \mathbb{R})_{\acute{e}t},\Lambda)$ of the field of real numbers with a $\mathbb{Z}/2\mathbb{Z}$-equivariant Witt theory of finitely generated projective $\Lambda$-modules is an incentive to study constructible Witt theory.
We have shown an isomorphism 
$W^i_c((\spec\mathbb{R})_{\acute{e}t}, \Lambda)=W^i_{lf}(\Lambda[\mathbb{Z}/2\mathbb{Z}]),$
where $W^i_{lf}(\Lambda[\mathbb{Z}/2\mathbb{Z}])$ denotes a $\mathbb{Z}/2\mathbb{Z}$-equivariant Witt theory of finitely generated projective $\Lambda$-modules.

For a smooth complex algebraic variety $X$ and finite coefficient ring $\Lambda$, we can identify the algebraically defined constructible Witt theory of $X$ with the constructible Witt theory of the underlying complex manifold $X^{an}$ of the complex points of $X$ in Euclidean topology, defined topologically in \cite{woolf2008witt}. The triangulated categories with duality used in the two constructions are equivalent and induce an isomorphism of the constructible Witt theory of $X$ with the (topological) constructible Witt theory of $X^{an}$.

Using proper pushforwards in the constructible Witt theory, we have defined an algebraic analog of signature considered in \cite{woolf2008witt}; see the Theorem \ref{algebraic analog of signature}. For instance, for a projective real algebraic variety $X$ with the structure morphism $f:X\rightarrow \spec \mathbb{R}$ and $\Lambda$ a ring of finite characteristic not equal to 2, we have the induced proper pushforward for constructible Witt theory 
\[W^{i}(f_*): W^i_c(X_{\acute{e}t}, \Lambda)\rightarrow W^i_c((\spec\mathbb{R})_{\acute{e}t}, \Lambda)=W^i_{lf}(\Lambda[\mathbb{Z}/2\mathbb{Z}]).\]

It can be observed that the theory of \'etale sheaves on $\spec \mathbb{R}$ is the only one that can be related to an equivariant theory frequently studied, namely, the $\mathbb{Z}/2\mathbb{Z}$-equivariant theory since this is the only case in which the absolute Galois group is a familiar one (the cyclic group of order 2). In this paper we identify the constructible Witt theory of sheaves of $\Lambda$-modules for $\spec \mathbb{R}$ with the Witt theory of the group ring $\Lambda[\mathbb{Z}/2\mathbb{Z}]$.
In general, over a field $k$ having characteristic prime to the characteristic of the ring $\Lambda$ the groups $W^i_c((\spec k)_{\acute{e}t},\Lambda)$ should be related to a Witt theory of finitely generated projective $\Lambda$-modules with an action of the absolute Galois group $Gal(k^{sep}/k)$ (a profinite group) of the separable closure of $k$. Even with widening interest in equivariant theories we are not aware of any studies in Witt theory for actions of pro-finite groups and the work in this paper should encourage more research on this.

Now we describe the contents of the paper in details. In section \ref{preliminaries} we quickly recall the Witt theory of triangulated categories with duality. In section \ref{Derived categories of constructible sheaves with tor-finite cohomology} the derived category $D^b_{ctf} (X_{\acute{e}t},\Lambda)$ of \'etale sheaves of  $\Lambda$-modules having finite Tor-dimensions and constructible cohomology sheaves is described as a triangulated category with duality. Under appropriate restrictions these categories have been widely studied for a long time in \'etale cohomology theory. In this paper we use the recent advances made  by Cisinski and D\'eglise on  \'etale motives to understand the category $D^b_{ctf} (X_{\acute{e}t},\Lambda)$ by using an equivalent model $DM_{h,lc}(X,\Lambda)$ of locally constructible $h$-motives. The categories $DM_{h,lc}(X,\Lambda)$ offer a greater generality for the six-functor formalism and also a roadmap for studying the relations between constructible Witt groups and the rational motivic cohomology. In section \ref{Derived categories of constructible sheaves with tor-finite cohomology} we use the work of Cisinski and D\'eglise to recast the relevant results for the bounded derived category $D^b_{ctf} ((-)_{\acute{e}t},\Lambda)$ and using these in section \ref{constructible Witt theory} we develop the constructible Witt theory.

\begin{theorem}
Let $B$ be an excellent noetherian scheme of dimension $\leq 2$ and $\Lambda$ a noetherian  ring of positive characteristic prime to the residue characteristics of $B$. Let $\phi:S\rightarrow B$ be a regular separated finite type $B$-scheme and $f:X\rightarrow S$ a separated morphism of finite type. The categories $D^b_{ctf}(X_{\acute{e}t},\Lambda)\subset D^b(X_{\acute{e}t},\Lambda)$ are closed under the six-functors of Grothendieck in $D^b(X_{\acute{e}t},\Lambda)$. 
For a $\otimes$-invertible object $T\in D^b_{ctf}(S_{\acute{e}t},\Lambda)$ the functor 
$$D_X(T)=R\mathcal{H}om (-, f^!(T)): D^b_{ctf}(X,\Lambda)^{op}\rightarrow D^b_{ctf}(X,\Lambda)$$
is a duality functor on $D^b_{ctf}(X_{\acute{e}t},\Lambda)$. In particular, we have triangulated categories with duality $(D^b_{ctf}(B_{\acute{e}t},\Lambda), R\mathcal{H}om(-, \Lambda))$ and
$(D^b_{ctf}(S_{\acute{e}t},\Lambda), R\mathcal{H}om(-, \Lambda))$
\end{theorem}

\begin{theorem}[Pullbacks for \'etale morphisms and proper pushforwards] Under  the same assumptions on $B$, $\Lambda$ and $S$ as in the theorem above, let $f: X\rightarrow S$ and $g: Y\rightarrow S$ be separated $S$-schemes of finite type, and $T$ be a $\otimes$-invertible object in $D^b_{ctf}(S_{\acute{e}t},\Lambda)$. Then an \'etale morphism $h:X\rightarrow Y$ of $S$-schemes induces a morphism of triangulated categories with duality 
\[h^*: (D^b_{ctf}(Y_{\acute{e}t},\Lambda), g^!T) \rightarrow 
(D^b_{ctf}(X_{\acute{e}t},\Lambda), f^!T)\]
and homomorphisms 
$h^*:W^i_c(Y_{\acute{e}t}, \Lambda)\rightarrow W^i_c(X_{\acute{e}t},\Lambda)$
of constructible Witt groups. If the morphism $h:X\rightarrow Y$ is proper, then the pushforward morphism 
\[h_*: (D^b_{ctf}(X_{\acute{e}t},\Lambda), f^!T) \rightarrow 
(D^b_{ctf}(Y_{\acute{e}t},\Lambda), g^!T)\]
is a morphism of triangulated categories with duality and induces 
pushforward morphisms on constructible Witt groups $h_*:W^i_c(X_{\acute{e}t}, \Lambda)\rightarrow W^i_c(Y_{\acute{e}t},\Lambda)$.
\end{theorem}

In the section \ref{constructible Witt theory of fields} we provide the description of constructible Witt theory of the field $\mathbb{R}$ of real numbers in terms of an equivariant Witt theory of finitely generated $\Lambda$-modules with Galois action. 

\begin{theorem} For a ring $\Lambda$ of finite characteristic not equal to $2$  the bounded derived category $D^b_{ctf}((\spec \mathbb{R})_{\acute{e}t}, \Lambda)$ is equivalent to the bounded derived category $D^b(Proj(\Lambda[\mathbb{Z}/2\mathbb{Z}]))$ of finitely generated projective $\Lambda$-modules with an action of the absolute Galois group $Gal(\mathbb{C}/\mathbb{R})=\mathbb{Z}/2\mathbb{Z}$. There is an identification of constructible Witt groups $W^i_c((\spec\mathbb{R})_{\acute{e}t}, \Lambda)$ with $\mathbb{Z}/2\mathbb{Z}$-equivariant Witt groups of finitely generated projective $\Lambda$-modules.
\end{theorem}

In section \ref{section 6} we explore the relationship between topological and algebraic constructible Witt theories for smooth complex algebraic varieties. In Theorem \ref{Top-Alg} we prove the following.

\begin{theorem}
Let $X$ be a smooth algebraic variety over $\mathbb{C}$ and $\Lambda$ be a finite ring with characteristic not equal to $2$. Then we have an equivalence of triangulated categories with duality
$$(D^b_{ctf}(X_{\acute{e}t},\Lambda), R\mathcal{H}om(-,\Lambda) \xrightarrow{\simeq} (D^b_c(X^{an},\Lambda),R\mathcal{H}om(-,\Lambda))$$
and isomorphism
\[W^i_c(X_{\acute{e}t}, \Lambda)\xrightarrow{\ \sim \ } W^i_c(X^{an}, \Lambda)\]
of the constructible Witt theory of the scheme $X$ with the constructible Witt theory of the topological space $X^{an}$ defined in 
\cite{woolf2008witt}.
\end{theorem}

In section \ref{signature} Theorem
\ref{algebraic analog of signature}, we provide a signatures 
with values in appropriate Witt groups of finitely generated locally free modules: For projective real algebraic varieties, a signature with values in
$W^i_{lf}(\Lambda[\mathbb{Z}/2\mathbb{Z}])$; and, 
for a projective complex algebraic varieties, a signature with values in 
$W^i_{lf}(\Lambda)$. 
An algebraic cobordism interpretation of these signatures is a work in progress.

{\it Conventions.} In this paper the coefficient ring $\Lambda$ will be assumed to be of finite characteristic not equal to $2$ and all the schemes considered will have the property that their residue characteristics are prime to the characteristic of $\Lambda$. Whenever we consider complex algebraic varieties, we may further restrict $\Lambda$ to be a finite ring.   All the schemes will be assumed to be noetherian and of finite Krull dimension.

{\it Acknowledgements.} We are thankful to K. Arun Kumar for helpful discussions during the preparation of this manuscript.

\section{Witt theory of triangulated categories with duality}
\label{preliminaries}
In this section we recall the basic theory of Witt groups for triangulated categories with duality developed by Balmer in \cite{balmer1999derived}, \cite{balmer2000triangular}, \cite{balmer2001triangular}.

\subsection{Triangulated categories with duality} 
Let $\delta =\pm 1$. For triangulated categories $(K_1,T_1)$ and $(K_2, T_2)$, a contravariant additive functor $F:K_1\rightarrow K_2$ is said to be $\delta$-exact if $T_{2}^{-1}\circ F = F\circ T_{1}$ and if for any exact triangle 
    $$A\xrightarrow{u} B\xrightarrow{v} C\xrightarrow{w}T_1(A)$$
    the following triangle is exact:
    $$F(C) \xrightarrow{F(v)} F(B)\xrightarrow{F(u)} F(A)\xrightarrow{\delta \cdot T_2(F(w))}T_2(F(C))$$

\begin{definition} Let $(K, T)$ be a triangulated category and $\delta =\pm1$. We will always assume that $1/2\in K$. A \emph{$\delta$-duality} is a $\delta$-exact contravariant functor $\#:K\rightarrow K$ such that there exist an isomorphism $\omega : Id\xrightarrow{\simeq} \# \circ \#$ satisfying the conditions: $$\omega_{T(M)} = T(\omega_M) \hspace{0.5cm} \text{and} \hspace{0.5cm} (\omega_M)^{\#} \circ \omega_{M^\#} = Id_{M^{\#}}$$ for any object $M$ of $K$. Then the triple $(K, \:\#,\: \omega )$ is called a $triangulated$ $category$ $with$ $\delta$-$duality$.

In case $\delta=1$, we shall talk about a $duality$ and in case $\delta=-1$, we shall use the word $skew$-$duality$.
    
\end{definition}
\begin{example}
Let  $(K,\:\#,\:\omega)$ be a tringulated category with $\delta$-duality $(\delta=\pm1)$.
\begin{enumerate}
\item Let $n\in \mathbb{Z}$. Then $(K,\:T^{n}\circ \#,\:\omega)$ is again a triangulated category with $((-1)^{n}\cdot \delta)$-duality.
\item Also, $(K,\:\#,\:-\omega)$ is again a triangulated category with $\delta$-duality.
\end{enumerate}
\end{example}

Consider a triangulated category  $(K,\:\#,\:\omega)$ with $\delta$-duality $(\delta=\pm 1)$. A \emph{symmetric space} is a pair $(P,\phi)$ such that $P$ is an object in $K$ and $\phi:P \xrightarrow{\simeq}P^{\#}$  is an isomorphism such that $\phi^{\#} \circ \omega_{P}=\phi$.
A \emph{skew-symmetric} $form$ is a pair $(P,\phi)$ such that $P$ is an object in $K$ and $\phi:P \xrightarrow{\simeq}P^{\#}$  is an isomorphism such that $\phi^{\#} \circ \omega_{P}=-\phi$. A skew-symmetric form in $K$ is a symmetric form in $(K,\#,-\omega)$. Orthogonal sum and isometries are defined as usual.

\begin{remark}Let $(K,\:\#,\:\omega)$ be a triangulated category with $\delta$-duality  $(\delta=\pm 1)$. Then as suggested by the above example, the \emph{translated (or shifted)} structure of triangulated category with $(-\delta)$-duality is 
$$T(K,\:\#,\:\omega)\coloneqq (K,\:T\circ \#,\:(-\delta)\cdot \omega).$$ 
Also, if $(K,\:\#,\:\omega)$ is a triangulated category with duality (i.e. $\delta=+1$) then $$T^n (K,\:\#,\:\omega)=(K,\:T^n \circ \#,\:(-1)^{n(n+1)/2}\cdot \omega)$$ is a triangulated category with $(-1)^n$-duality, for all $n\in \mathbb{Z}$.
\end{remark}

\subsection{Witt groups of triangulated categories with duality}
In his 1937 paper \cite{witt1937theorie}, Ernst Witt introduced a group structure and even a ring
structure on the set of isometry classes of anisotropic quadratic forms, over an
arbitrary field k. This object is now called the Witt group W(k) of k. Since then, Witt’s construction has been generalized and extended in many ways. In this paper we work in the setting of triangulated categories with duality developed by Balmer. 

\begin{definition}
Let $(K,\:\#,\:\omega)$ be a triangulated category with $\delta$-duality (for $\delta=\pm1$) and $1/2\in K$. Let $(P,\phi)$ be a symmetric space. A pair $(L,\alpha)$ where $L$ is an object of $K$ and $\alpha: L\rightarrow P$ is a morphism, is called a \emph{sublagrangian} of $(P,\phi)$ if $\alpha^{\#}\phi \alpha=0$. A triple $(L,\alpha, w)$ is called a \emph{lagrangian} if the following triangle is exact:
    $$T^{-1}(L^\#)\xrightarrow{w}L\xrightarrow{\alpha}P\xrightarrow{\alpha^{\#}\phi}L^\# \hspace{0.5cm}\text{and} \hspace{0.5cm} \text{if $w$ is $\delta$-symmetric, i.e.}\: T^{-1}(w^\#)=\delta \cdot w.$$
\end{definition}
\begin{definition}
Let $(K,\:\#,\:\omega)$ be a triangulated category with $\delta$-duality (for $\delta=\pm1$). A symmetric space $(P,\:\phi)$ is \emph{neutral} or \emph{metabolic} if it possesses a lagrangian.	
\end{definition}
\begin{definition}
	Let $(K,\:\#,\:\omega)$ be a triangulated category with $\delta$-duality (for $\delta=\pm1$). We define the \emph{Witt Monoid} of $K$ to be the monoid of isometry classes of symmetric spaces endowed with the diagonal sum and we denote it by $$MW(K,\:\#,\:\omega).$$
The set of isometry classes of neutral spaces forms a submonoid $NW(K,\:\#,\:\omega)\subset MW(K,\:\#,\:\omega)$. Then quotient monoid is a group, called the \emph{Witt Group} of $(K,\:\#,\:\omega)$ and written as $W(K,\:\#,\:\omega)$. Thus,
$$W(K,\:\#,\:\omega) = \frac{MW(K,\:\#,\:\omega)}{NW(K,\:\#,\:\omega)}.$$
If $(P,\phi)$ is symmetric space, we write $[P,\phi]$ for its class in the corresponding Witt group. We say that two symmetric spaces are \emph{Witt-equivalent} if their classes in the Witt group are the same. We also define
$$W^n (K,\:\#,\:\omega)\coloneqq W(T^n (K,\:\#,\:\omega))$$
for all $n\in \mathbb{Z}$. These are the $shifted $ Witt groups of $(K,\:\#,\:\omega)$.	
\end{definition}

\subsection{Functoriality}
Let $(K_1, \#_{1}, \omega_{1})$ and $(K_2, \#_{2},  \omega_{2})$ be triangulated categories with duality. Assume that $\#_1$ and $\#_2$ are both either exact or skew exact. A \emph{morphism} of triangulated categories with duality refers to a covariant additive functor $F: K_{1}\rightarrow K_{2}$ that preserves exactness or skew-exactness, and satisfies the following conditions:
$$F\circ \#_1 = \#_2\circ F \hspace{0.5cm}\text{and}\hspace{0.5cm} F(\omega_1)=\omega_2.$$
In this case, if $(P,\phi)$ is a symmetric space for $\#_1$ then $(F(X),F(\phi))$ is a symmetric space for $\#_{2}$. If $(L,\alpha,w)$ is a lagrangian of the starting space, then $(F(L),F(\alpha),\delta \cdot F(w))$ is a lagrangian of its image, where $\delta=\pm 1$ comes from $\delta$-exactness of $F$. This tells that neutral forms are mapped to neutral forms. Hence $F$ induces group homomorphisms
$$W^n(F): W^n(K_1, \#_{1}, \omega_{1}) \rightarrow W^n(K_1, \#_{2}, \omega_{2})$$
for any $n\in \mathbb{Z}$.

\section{Derived category of constructible sheaves}\label{Derived categories of constructible sheaves with tor-finite cohomology}

In this section we recall some well-known results for the full triangulated subcategory 
$D^b_{ctf}(X_{\acute{e}t},\ \Lambda)$ of the derived category  $D(X_{\acute{e}t},\Lambda)$ of \'etale sheaves of 
$\Lambda$-modules on the small \'etale site $X_{\acute{e}t}$. We describe the category $D^b_{ctf}(X_{\acute{e}t},\Lambda)$ as a triangulated category with duality and discuss localization sequences. In the next section applying the general machinery of Balmer the constructible Witt-theory of a scheme $X$ with $\Lambda$-coefficients will be defined and studied via this triangulated category with duality. For the description of duality and localization we use the recent work of Cisinski and D\'eglise in \cite{cisinski2016etale}: The advantage of technical foundations in terms of \'etale motives is a greater generality of the results that otherwise will need more restrictive hypotheses within the framework of \cite{SGA4}, \cite{deligne1977seminaire} and \cite{illusie2014travaux}. For instance, as compared to the restriction of finiteness on coefficient rings one can develop the constructible Witt theory for `good enough' coefficient rings as defined in \cite[Def. 6.3.6]{cisinski2016etale}- this includes $\mathbb{Z}$ and any noetherian ring of positive characteristic. Another generality offered by the work of Cisinski and D\'eglise is extension of the theory to a more general base scheme as recalled in Theorem \ref{6-functor formalism for locally constructible h-motives} in terms of the subcategory $DM_{h,lc}(X,\Lambda)$ of locally constructible $h$-motives of the derived category $DM_{h}(X,\Lambda)$ of $h$-motives.

\subsection{Constructible sheaves of $\Lambda$-modules}
We begin by recalling the category to be used for defining constructible Witt theory- the bounded derived category $D^b_{ctf}(X_{\acute{e}t}, \Lambda)$ of \'etale sheaves having finite Tor-dimension and constructible cohomology sheaves. 
For easy reference we recall the small \'etale site $X_{\acute{e}t}$ of a scheme $X$ and the classically studied bounded derived category 
$D^b_{ctf}(X_{\acute{e}t}, \Lambda)$. As a category it is the category $\acute{E}t/X$ of \'etale $X$-schemes, and covering in the site $X_{\acute{e}t}$ are surjective families  $\{X_i^\prime \xrightarrow{\phi _i} X^\prime\}$ of morphisms in $\acute{E}t/X$.
Thus, 
$$
	\begin{cases}
	\: \: cat(X_{\acute{e}t})= \acute{E}t/X\\  \\  
	\: \: cov(X_{\acute{e}t})= \text{collection of surjective families of morpisms in } \acute{E}t/X. 
	\end{cases}
	$$
	\\
Sheaves on $X_{\acute{e}t}$ are called \emph{\'etale sheaves} on $X$.

Let $X$ be a noetherian scheme and $\Lambda$ a noetherian ring. A sheaf $\mathcal{F}$ of $\Lambda$-modules on $X_{\acute{e}t}$ is called $constructible$ if there exists a decomposition $\bigcup _{i=1} ^{n}X_i$ of $X$ into finitely many locally closed subsets $X_i$ such that each $\mathcal{F}|_{X_i}$ is locally constant and the stalks of $\mathcal{F}$ are finitely generated $\Lambda$-modules.
Along with constructibility the objects in the category $D^b_{ctf}(X_{\acute{e}t},\Lambda)$ have additional finiteness condition, namely that of having finite Tor-dimension.
Let $\mathcal{F}^\bullet $ be a bounded complex of sheaves of $\Lambda$-modules on $X$. We say that $\mathcal{F}^\bullet$ has \emph{finite Tor-dimension} if there exist an integer $n$ such that Tor$_{i}(\mathcal{F}^\bullet, M)=0$ for any $i>n$ and any constant sheaf of $\Lambda$-modules $M$ on $X$.

Let $X$ be a noetherian scheme and $\Lambda$ be a noetherian ring. Denote by $D(X_{\acute{e}t},\Lambda)$ the derived category of sheaves of $\Lambda$-modules over $X_{\acute{e}t}$. We get the full subcategories $D^\ast(X_{\acute{e}t},\Lambda)$ by taking $\ast=+,-,b$.

\begin{definition}[The category $D^b_{ctf}(X_{\acute{e}t},\Lambda)$ of sheaves having finite Tor-dimension and constructible cohomology sheaves]
The category $D^b_{ctf}(X_{\acute{e}t},\Lambda)\subset D(X_{\acute{e}t},\Lambda)$ is the full subcategory 
consisting of bounded complexes of sheaves of $\Lambda$-modules with finite Tor-dimension and having constructible cohomology sheaves. The objects in the category $D^b_{ctf}(X_{\acute{e}t},\Lambda)$ are complexes that are quasi-isomorphic in $D(X_{\acute{e}t}, \Lambda)$  to bounded complexes whose components are \emph{flat} and \emph{constructible} sheaves of $\Lambda$-modules \cite{fu2011etale}.
\end{definition}

\subsection{Duality on constructible sheaves}

As mentioned in the beginning of this section we now use the six-functor formalism for 
triangulated subcategories $DM_{h,lc}(X,\Lambda)$ of locally constructible $h$-motives of the derived category $DM_{h}(X,\Lambda)$ of $h$-motives from \cite{cisinski2016etale}  to the case of classically studied bounded derived categories $D^b_{ctf}(X_{\acute{e}t},\Lambda)$ of \'etale sheaves having finite Tor-dimension and constructible cohomology. For basic definitions and related material about the derived category of \'etale motives the original paper \cite{cisinski2016etale} by Cisinski and D\'eglise can be consulted.

In what follows recall that a scheme is said to be excellent if it has an open affine cover given by excellent rings: Examples of such rings include the ring $\mathbb{Z}$ of integers, fields, the ring $\mathbb{Z}_p$ of $p$-adic integers, finitely generated rings of the form $R[x_1,\cdots,x_r]/<f_1,\cdots,f_n>$ over an excellent ring $R$, and localizations of excellent rings. 

Compared to \cite{SGA4} and \cite{deligne1977seminaire} more general six-functor formalism for $D^b_{ctf}(X_{\acute{e}t},\Lambda)$ follows from the two results \cite[Theorem 6.3.11, Corollary 6.3.15]{cisinski2016etale} restated in the following theorem.

\begin{theorem}
\label{6-functor formalism for locally constructible h-motives}
Let $B$ be an excellent scheme of dimension $\leq 2$ and $\Lambda$ a noetherian  ring of positive characteristic prime to the residue characteristics of $B$. Let $\phi:S\rightarrow B$ be a regular separated finite type $B$-scheme and $f:X\rightarrow S$ a separated morphism of finite type. The categories $D^b_{ctf}(X_{\acute{e}t},\Lambda)\subset D^b(X_{\acute{e}t},\Lambda)$ are closed under the six-functors of Grothendieck in $D^b(X_{\acute{e}t},\Lambda)$. 
For a $\otimes$-invertible object $T\in D^b_{ctf}(S_{\acute{e}t},\Lambda)$ the functor 
$$D_X(T)=R\mathcal{H}om (-, f^!(T)): D^b_{ctf}(X_{\acute{e}t},\Lambda)^{op}\rightarrow D^b_{ctf}(X_{\acute{e}t},\Lambda)$$
is a duality functor on $D^b_{ctf}(X_{\acute{e}t},\Lambda)$. In particular, we have triangulated categories with duality $(D^b_{ctf}(B_{\acute{e}t},\Lambda), R\mathcal{H}om(-, \Lambda))$ and
$(D^b_{ctf}(S_{\acute{e}t},\Lambda), R\mathcal{H}om(-, \Lambda))$.
\end{theorem}
\begin{proof} Under the assumptions in the theorem the
triangulated subcategory $DM_{h,lc}(X,\Lambda)\subset DM_{h}(X,\Lambda)$ of locally constructible $h$-motives is closed under the six-functors of Grothendieck in $DM_{h}(X,\Lambda)$ and for 
a locally constructible $\otimes$-invertible object $U\in DM_{h}(S,\Lambda)$, the functor 
$$D_X(U)=R\mathcal{H}om (-, f^!(U)): DM_{h,lc}(X,\Lambda)^{op}\rightarrow DM_{h,lc}(X,\Lambda)$$
is a duality functor on $DM_{h,lc}(X,\Lambda)$. Since the equivalence of categories 
\[D^b_{ctf}(X_{\acute{e}t},\Lambda)\xrightarrow{\simeq}  DM_{h,lc}(X,\Lambda) \]
in \cite[Theorem 6.3.11]{cisinski2016etale}
is compatible with the six functor formalism in the two settings, the proof of the theorem is  
complete. 
\end{proof}
\noindent {\it Notation.} With the notations of the above theorem the triangulated category $D^b_{ctf}(X_{\acute{e}t},\Lambda)$ with the duality $D_X(T)$ will be denoted by $(D^b_{ctf}(X_{\acute{e}t},\Lambda), D_X(T))$ or $(D^b_{ctf}(X_{\acute{e}t},\Lambda), f^!T)$.

\begin{proposition}[Pullbacks for \'etale morphisms and pushforwards for proper morphisms]
\label{derived category pullbacks and pushforwards} 
Under  the same assumptions on $B$, $\Lambda$ and $S$ as in the theorem above, let $f: X\rightarrow S$ and $g: Y\rightarrow S$ be separated $S$-schemes of finite type, and $T$ be a $\otimes$-invertible object in $D^b_{ctf}(S_{\acute{e}t},\Lambda)$. Let $h:X\rightarrow Y$ be a morphism of $S$-schemes. If $h$ is \'etale, then the \emph{pullback} morphism $h^*: D^b_{ctf}(Y_{\acute{e}t},\Lambda) \rightarrow 
D^b_{ctf}(X_{\acute{e}t},\Lambda)$ is a morphism of triangulated categories with duality 
\[h^*: (D^b_{ctf}(Y_{\acute{e}t},\Lambda), g^!T) \rightarrow 
(D^b_{ctf}(X_{\acute{e}t},\Lambda), f^!T).\] 
If the morphism $h:X\rightarrow Y$ is proper, then the pushforward morphism $h_*:D^b_{ctf}(X_{\acute{e}t},\Lambda) \rightarrow D^b_{ctf}(Y_{\acute{e}t},\Lambda)$ is a morphism of triangulated categories with duality
\[h_*: (D^b_{ctf}(X_{\acute{e}t},\Lambda), f^!T) \rightarrow 
(D^b_{ctf}(Y_{\acute{e}t},\Lambda), g^!T).\]
\end{proposition}

\begin{proof} We explain the details for the pushforwards $h_*$ for proper morphisms of schemes and leave the part on $h^*$ as an exercise. We show that in the following diagram of triangulated categories 
\[
\xymatrix{
D_{ctf}^b(X_{\acute{e}t},\Lambda)^{op} \ar[rr]^{\ \ \ D_X(T)\ \ \ } \ar[d]^{h_*} && D_{ctf}^b(X_{\acute{e}t},\Lambda) \ar[d]^{h_*}\\
D_{ctf}^b(Y_{\acute{e}t},\Lambda)^{op} \ar[rr]^{\ \ \ D_Y(T)\ \ \ } &&D_{ctf}^b(Y_{\acute{e}t},\Lambda)}
\]
the two compositions are naturally isomorphic by using the relations among the functors in this diagram. We have $f=g\circ h$. Given an element $F\in D_{ctf}^b(X_{\acute{e}t},\Lambda)$ we have the following natural identifications 
\[
\xymatrix{
h_*(D_X(T)(F))  = h_*(R\mathcal{H}om (F, f^!(T))) \simeq h_*(R\mathcal{H}om (F, h^!g^!(T)))\\
\simeq R\mathcal{H}om_Y (h_!F, f^!T) \simeq R\mathcal{H}om_Y (h_*F, f^!T)
= D_Y(T)(h_*(F)))
}\]
where in second to the last identification we use the fact that $h_!=h_*$ for the proper morphism $h$.\end{proof}

\subsection{Localization}  We know that for a closed immersion $i:Z\hookrightarrow X$ of schemes the induced exact functor $i_*:Shv(Z_{\acute{e}t},\Lambda)\rightarrow Shv(X_{\acute{e}t},\Lambda)$ is fully faithful and its essential image is the subcategory of those sheaves whose support is contained in $Z$. The natural transformation $Ri_*=i_*:D(Z_{\acute{e}t},\Lambda)\rightarrow D(X_{\acute{e}t},\Lambda)$ of the derived categories induces an equivalence of $D(Z_{\acute{e}t},\Lambda)$ with $D_Z(X_{\acute{e}t},\Lambda)$- the strictly full saturated subcategory of $D(X_{\acute{e}t},\Lambda)$ consisting of complexes whose cohomology sheaves are supported in $Z$. Let $j:U=X-Z\hookrightarrow X$ be the complementary open immersion and $j^*:Shv(X_{\acute{e}t},\Lambda)\rightarrow Shv(U_{\acute{e}t},\Lambda)$ the induced map of sheaves.
Then we have an exact sequence of triangulated categories
\begin{equation}
\label{basic localization sequence}
D(Z_{\acute{e}t},\Lambda)\xrightarrow{i_*} D(X_{\acute{e}t},\Lambda)\xrightarrow{j^*} D(U_{\acute{e}t},\Lambda)\end{equation}
in which the $D(Z_{\acute{e}t},\Lambda)$ can be identified with $D_Z(X_{\acute{e}t},\Lambda)$.
Let $D^b_{ctf,Z}(X_{\acute{e}t},\Lambda)\subset D_{ctf}^b(X_{\acute{e}t},\Lambda)$ be the strictly full saturated triangulated subcategory consisting of bounded complexes having constructible and finite Tor-dimension whose cohomology sheaves are supported in $Z$. Under the assumptions in theorem \ref{6-functor formalism for locally constructible h-motives} for the scheme $X$ and the ring $\Lambda$ the six-functor formalism for \'etale motives in \cite[Corollary 6.3.15]{cisinski2016etale} implies that the natural equivalence of the categories $D(Z_{\acute{e}t},\Lambda)$ with $D_Z(X_{\acute{e}t},\Lambda)$  induces an equivalence of $D_{ctf}^b(Z_{\acute{e}t},\Lambda)$ with $D^b_{ctf,Z}(X_{\acute{e}t},\Lambda)$.

\begin{lemma} 
\label{duality on induced morphisms via closed and open immersions}
With the same assumptions as in Theorem $\mathrm{\ref{6-functor formalism for locally constructible h-motives}}$ for the schemes $B,\ S,\ X$, 
the ring $\Lambda$ and the map $f:X\rightarrow S$, let $i:Z\hookrightarrow X$ be a closed immersion and $j:U=X-Z\hookrightarrow X$ the complementary open immersion. The induced functors 
$i_*:D_{ctf}^b(Z_{\acute{e}t},\Lambda)\rightarrow D_{ctf}^b(X_{\acute{e}t},\Lambda)$ and 
$j^*:D_{ctf}^b(X_{\acute{e}t},\Lambda) \rightarrow D_{ctf}^b(U_{\acute{e}t},\Lambda)$
are natural transformations of triangulated categories with duality
\[
\xymatrix{
(a.)\ \ \ \ \ \ \ \ \ \ \ \ \ i_*:(D_{ctf}^b(Z_{\acute{e}t},\Lambda), D_Z(T)) \rightarrow (D_{ctf}^b(X_{\acute{e}t},\Lambda), D_X(T)) \\
(b.)\ \ \ \ \ \ \ \ \ \ \ \ \ j^*:(D_{ctf}^b(X_{\acute{e}t},\Lambda), D_X(T)) \rightarrow (D_{ctf}^b(U_{\acute{e}t},\Lambda), D_U(T)) 
}\]
for every $\otimes$-invertible object $T\in D^b_{ctf}(S_{\acute{e}t},\Lambda)$.
\end{lemma}
\begin{proof} This is a special case of Proposition \ref{derived category pullbacks and pushforwards}.
\end{proof}

Under the equivalence of $D_{ctf}^b(Z_{\acute{e}t},\Lambda)$ with $D^b_{ctf,Z}(X_{\acute{e}t},\Lambda)$ we will consider $D^b_{ctf,Z}(X_{\acute{e}t},\Lambda)$ as a triangulated category with duality and denote the duality by $D_Z(T)$ for a $\otimes$-invertible object $T\in D^b_{ctf}(S_{\acute{e}t},\Lambda)$.
The localization property of the categories $D^b_{ctf}((-)_{\acute{e}t},\Lambda)$ is the following exact sequence of triangulated categories with the notations of this subsection.

\begin{theorem}
\label{localization sequence of triangulated categories with duality}
Let $B$ be an excellent scheme of dimension $\leq 2$ and $\Lambda$ a noetherian  ring of positive characteristic prime to the residue characteristics of $B$. Let $\phi:S\rightarrow B$ be a regular separated finite type $B$-scheme and $f:X\rightarrow S$ a separated morphism of finite type. The exact sequence $\mathrm{\ref{basic localization sequence}}$ of triangulated categories induces a localization sequence of triangulated categories with duality
\begin{equation}
\label{basic localization sequence with duality}
(D^b_{ctf,Z}(X_{\acute{e}t},\Lambda), D_Z(T)) \xrightarrow{i_*} (D^b_{ctf}(X_{\acute{e}t},\Lambda), D_X(T))\xrightarrow{j^*} D^b_{ctf}(U_{\acute{e}t},\Lambda), D_U(T))
\end{equation}
for a $\otimes$-invertible object $T\in D^b_{ctf}(S_{\acute{e}t},\Lambda)$.
\end{theorem}
\begin{proof}
Using Grothendieck's six-functor formalism for $DM_{h,lc}((-)_{\acute{e}t},\Lambda)$ in \cite[Corollary 6.3.15]{cisinski2016etale} and the identification of $D^b_{ctf}((-)_{\acute{e}t},\Lambda)$ with $DM_{h,lc}((-)_{\acute{e}t},\Lambda)$ in  \cite[Theorem 6.3.11]{cisinski2016etale} we observe that the morphisms of triangulated categories in \ref{basic localization sequence with duality} define a localizing sequence. In view of Lemma \ref{duality on induced morphisms via closed and open immersions} the proof of the theorem is complete.
\end{proof}

\section{Witt theory of constructible sheaves}\label{constructible Witt theory}
In this section we now define the constructible Witt groups of \'etale sheaves of $\Lambda$-modules on a noetherian scheme $X$ as Balmer's Witt groups of the triangulated category with duality  $D^b_{ctf}(X_{\acute{e}t},\Lambda)$ discussed in the 
section \ref{Derived categories of constructible sheaves with tor-finite cohomology}. We describe some of their basic properties. 
We will work under the assumptions as in Theorem \ref{6-functor formalism for locally constructible h-motives} for the schemes $B$, $S$, $X$, the ring $\Lambda$ and the map $f:X\rightarrow S$.

\begin{definition}[Constructible Witt groups]
The category $(D^b_{ctf}(X_{\acute{e}t},\Lambda),D_{X}(T))$
is a triangulated category with duality for a $\otimes$-invertible object $T\in D^b_{ctf}(S_{\acute{e}t},\Lambda)$. 
The constructible Witt groups of sheaves of $\Lambda$-modules on $X$ are Witt groups of this triangulated category with duality. We will denote these by $W^n_c(X_{\acute{e}t}, \Lambda, T)$. In notations we may suppress the mention of the $\otimes$-invertible object $T$ and use the notation  $W^n_c(X_{\acute{e}t}, \Lambda)$.
\end{definition}

\subsection{Functoriality for constructible Witt groups }
\label{functoriality of dbctf}

We can directly use Proposition \ref{derived category pullbacks and pushforwards} to get the functoriality of constructible Witt groups. Consider the same assumptions on $B$, $X$, $\Lambda$ and $S$ as in Proposition \ref{derived category pullbacks and pushforwards} and $T$ a $\otimes$-invertible object in $D^b_{ctf}(S_{\acute{e}t},\Lambda)$. 

For an \'etale morphism $h:X\rightarrow Y$ of $S$-schemes, we get a morphism of triangulated categories with duality
\[h^*: (D^b_{ctf}(Y_{\acute{e}t},\Lambda), g^!T) \rightarrow 
(D^b_{ctf}(X_{\acute{e}t},\Lambda), f^!T).\]
Then $h^\ast$ induces a group homomorphism of constructible Witt groups
$$W^{n}(h^\ast):W^n_c(Y_{\acute{e}t},\Lambda)\rightarrow W^n_c(X_{\acute{e}t},\Lambda)$$
for all $n\in \mathbb{Z}$.

Also, when $h: X\rightarrow Y$ is a proper morphism of $S$-schemes, we get a morphism of triangulated categories with duality   
\[h_\ast: (D^b_{ctf}(X_{\acute{e}t},\Lambda), f^!T) \rightarrow 
(D^b_{ctf}(Y_{\acute{e}t},\Lambda), g^!T).\]
Then $h_\ast$ induces a group homomorphism of constructible Witt groups
$$W^{n}(h_\ast):W^n_c(X_{\acute{e}t},\Lambda)\rightarrow W^n_c(Y_{\acute{e}t},\Lambda)$$
for all $n\in \mathbb{Z}$.

\subsection{Localization sequence for constructible Witt groups}
Now we describe the localization sequence for the constructible Witt theory.
Recall that a \emph{localization of triangulated categories}
\begin{equation}\label{Localization of triangulated categories}
    J\xrightarrow{j} K\xrightarrow{q} L
\end{equation}
is an exact sequence of triangulated categories, that is, $L$ is a localization of $K$ with respect to a saturated class of morphisms $S$, and the triangulated category $J$ is the full subcategory of $K$ consisting of objects $X \in K$ for which $q(X) = 0$ in $L$.

If $(K, \#,\omega)$ is a triangulated category with duality and the class of morphisms $S$ is \emph{compatible} with duality on $K$ i.e. $\#(S)=S$, then $J$ endowed with restriction of $\#$ and $S^{-1}K$ endowed with the localization of $\#$ are triangulated categories with duality and the exact natural transformations $j$ and $q$ are morphisms of triangulated categories with duality. Such a localization sequence of triangulated categories with duality gives a 12-term exact sequence of corresponding Witt groups.

The localization sequence for the constructible Witt theory is a consequence of the localization sequence
of triangulated categories with duality in Theorem \ref{6-functor formalism for locally constructible h-motives}. 
Given a closed immersion $i:Z\hookrightarrow X$ with  the complementary open immersion $j:U=X-Z \hookrightarrow X$, Theorem \ref{localization sequence of triangulated categories with duality} gives the localization of triangulated categories with duality
$$(D^b_{ctf}(Z_{\acute{e}t},\Lambda), D_Z(T)) \xrightarrow{i_*} (D^b_{ctf}(X_{\acute{e}t},\Lambda), D_X(T))\xrightarrow{j^*} D^b_{ctf}(U_{\acute{e}t},\Lambda), D_U(T)).$$
The corresponding 12-term exact sequence of constructible Witt groups can be written as 
 $$\cdots \rightarrow W_{c}^{n-1}(U_{\acute{e}t},\Lambda) \xrightarrow{\partial^{n-1}} W^n_c(Z_{\acute{e}t},\Lambda) \xrightarrow{W^{n}(i_\ast)} W^n_c(X_{\acute{e}t},\Lambda)$$
$$\xrightarrow{W^{n}(j^\ast)} W^n_c(U_{\acute{e}t},\Lambda) \xrightarrow{\partial^n} W^{n+1}_c(Z_{\acute{e}t},\Lambda)\rightarrow \cdots$$
using the 4-periodicity of triangular Witt groups.

\subsection{Homotopy Invariance} With the techniques used in this paper we are not in a position to prove homotopy invariance of constructible Witt theory simply because duality is not compatible with the six-functor formalism for $D^b_{ctf}((-)_{\acute{e}t},\Lambda)$ in the same generality. To be specific, under the running assumptions in this section, for a vector bundle $p:V\rightarrow X$, the pullback functor $p^*:D^b_{ctf}(X_{et},\Lambda) \rightarrow D^b_{ctf}(V_{et},\Lambda) $ is fully faithful since the unit of adjunction $1\rightarrow p_*p^*$ is an isomorphism but $p_*$ is not compatible with dualities. It is instinctive to expect homotopy invariance for this theory, but then one should develop the theory in a more flexible setting for six-functor formalism - perhaps in the setting of infinity-categories.

 \section{Constructible Witt theory of $\spec  \mathbb{R}$}
 \label{constructible Witt theory of fields}
In this section we will identify the constructible Witt theory of \'etale sheaves of $\Lambda$-modules on $\spec \mathbb{R}$, for a ring $\Lambda$ of finite characteristic not equal to 2, as a $\mathbb{Z}/2\mathbb{Z}$-equivariant Witt theory of $\Lambda$.   Also for a real projective variety $f:X\rightarrow \spec  \mathbb{R}$ we construct natural homomorphisms 
\[W^{i}(f_*): W^i_c(X, \Lambda)\rightarrow W^i_c(\mathbb{R}, \Lambda)=W^i_{lf}(\Lambda[\mathbb{Z}/2\mathbb{Z}])\] 
which define an algebraic version of signature for $X$.

Let $X=\spec k$ be the affine scheme defined by a field  $k$. Let $k^{sep}$ be a separable closure of $k$ and let $Gal(k^{sep}/k)$ be the Galois group of $k^{sep}/k$ equipped with the canonical structure of a profinite group.
For each $k$-scheme $X^\prime$ we denote by $X^\prime(k^{sep})$ the set of $k^{sep}$-valued points on $X^\prime/k$, i.e. the set of $k$-morphisms $\spec k^{sep}\rightarrow X^\prime$. A $k^{sep}$-valued point of $X^\prime$ corresponds uniquely to a point $x^\prime \in X^\prime$ together with a $k$-homomorphism $k(x^\prime)\rightarrow k^{sep}$. 
The following equivalence of the category of \'etale sheaves of $\Lambda$-modules on $\spec k$ with the category of 
$\Lambda$-modules with continuous $Gal(k^{sep}/k)$-action for the profinite topology on the Galois group is well-known. 

\begin{theorem}\label{Most useful corollary}
For the scheme $\spec k$ the functor defined by taking stalks at the geometric point of $\spec k$ given by a choice $k\subset k^{sep}$ of the separable closure 
$$\mathcal{F}\mapsto \varinjlim_{k'} \ \mathcal{F}(\spec k^{\prime}),\ \ \ \ k\subset k'\subset k^{sep} \ and\ k'\ a\ finite\ extension\ of\ k$$
is an equivalence between the category of sheaves of $\Lambda$-modules on $(\spec k)_{\acute{e}t}$ and the category of continuous $\Lambda[G]$-modules for a group ring $\Lambda[G]$. Thus, in the case of the field $\mathbb{R}$ of real numbers we have an equivalence 
\[ \phi: Shv((\spec \mathbb{R})_{\acute{e}t}, \Lambda)\rightarrow 
Mod(\Lambda[\mathbb{Z}/2\mathbb{Z}]) \]
of the category of \'etale sheaves of $\Lambda$-modules on $\spec \mathbb{R}$ with the category of all $\mathbb{Z}/2\mathbb{Z}$-equivariant $\Lambda$-modules.
\end{theorem}
\begin{proof}See \cite[Corollary 2.2]{tamme2012introduction} for details.
\end{proof}

\subsection{Flat and constructible sheaves on 
$(\spec \mathbb{R})_{\acute{e}t}$}
\label{Case of flat and constructible}
Let $\mathcal{F}$ be a flat and constructible sheaf of $\Lambda$-modules on $(\spec \mathbb{R})_{\acute{e}t}$. As $\mathcal{F}$ is a flat sheaf, the stalk of $\mathcal{F}$ at the single geometric point in $\spec\mathbb{R}$ given by $\spec\mathbb{C}$  is a flat $\Lambda$-module. 
Also $\mathcal{F}$ being a constructible sheaf of $\Lambda$-modules, it is locally constant and its stalk $\mathcal{F}_x$ is finitely generated $\Lambda$-module. 
Thus, for a flat and constructible sheaf  $\mathcal{F}$ on $(\spec \mathbb{R})_{\acute{e}t}$ the stalk $\mathcal{F}_x$
is a flat and finitely generated $\Lambda$-module. 
In this paper the coefficient rings are assumed to be noetherian and therefore for a flat and constructible sheaf $\mathcal{F}$ the stalk  $\mathcal{F}_x$
being flat and finitely generated is a 
finitely generated projective $\Lambda$-module with a $G=\mathbb{Z}/2\mathbb{Z}$ action.

Under the assumption that the characteristic of $\Lambda$ is different from 2 it so happens that the category of finitely generated projective $\Lambda$-modules with $\mathbb{Z}/2\mathbb{Z}$-action is equivalent to the category of finitely generated projective $\Lambda[\mathbb{Z}/2\mathbb{Z}]$-modules. For a lack of reference known to us we will recall this useful fact in terms of equivariant vector bundles on schemes.

\begin{lemma}
\label{equivariant splitting}
For an affine $G$-scheme $X$, under the assumption that $|G|$ is invertible in $\Gamma(X,\ \mc{O}_X)$, every short exact sequence of equivariant vector bundles splits equivariantly.
\end{lemma}
\begin{proof} Given a short exact sequence $0\rightarrow \mathcal{E}\xrightarrow{i} \mathcal{L}\xrightarrow{\phi} \mathcal{M}\rightarrow 0$ of $G$-equivariant vector bundles on $X$,
choose a (non-equivariant) splitting $s:\mathcal{M}\rightarrow \mathcal{L}$ and observe that the morphism 
$$\bar{s}:\mathcal{M}\rightarrow \mathcal{L}, \ \ \ a\mapsto \frac{1}{|G|}\sum_{g\in G}g\cdot s(g^{-1}\cdot a), \ \   a\in \mathcal{M}(U),\ U\subset X\ \mathrm{open}$$
of vector bundles is a $G$-equivariant splitting of $\phi$.  
\end{proof}
For an equivariant affine $G$-scheme $(\spec R,\tau)$, let $\rho_{_{R_\tau}}$ (or simply $\rho_{_R}$) be the skew group-ring $R_\tau G$ viewed as the twisted regular representation of $G$ over $R$. The $G$-action on $\rho_{_{R_\tau}}$ is given by \[g'\cdot(\Sigma_{g\in G}x_g\cdot e_g)= \Sigma_{g\in G}\tau_{g'}(x_g)\cdot e_{g'g},\] where $\{e_g:g\in G\}$ is the standard basis of $\rho_{_{R_\tau}}$
\begin{lemma}
\label{Equivalent Category of Summands in Regular Representations}
For an affine $G$-scheme $\spec R$, under the assumption that $|G|$ is invertible in $R$, every $G$-equivariant finitely generated projective $R$-module is a summand in a direct sum of regular representations.
\end{lemma}
\begin{proof} Let $\tau: G\rightarrow \mathrm{Aut}_{Sch/\spec \mathbb{Z}}(\spec R)$ denote the $G$-action on $\spec R$. Given an equivariant finitely generated projective $R$-module $P$ choose an epimorphism $\alpha: R^n\rightarrow P$. 
%The free rank 1 module $R$ is a $G$-equivariant summand in $\rho_{_{R_\tau}}$ via the inclusion 
%\[ R\longrightarrow \rho_{_{R_\tau}},\ \ \ x\mapsto \sum_{g\in G} \tau(x)e_g\]  
% and hence $R^n\hookrightarrow R^n\otimes_R \rho_{_{R_\tau}} = \rho_{_{R_\tau}}^n$ is a $G$-equivariant direct summand. 
The induced surjective map 
\[ \widetilde{\alpha}:\rho_{_{R_\tau}}^n \rightarrow P\ \ \ : \ \ \  \sum_{g\in G}x_g\cdot e_g \mapsto \sum_{g\in G}g\cdot \alpha(x_g)\]
is $G$-equivariant and we have an exact sequence of $G$-equivariant finitely generated projective $R$-modules 
$\mathrm{ker}\ \widetilde{\alpha}\xrightarrow{\ \ \ } \rho_{_{R_\tau}}^n \xrightarrow{\ \ \widetilde{\alpha}\ \ } P.$
Choosing an equivariant splitting by Lemma \ref{equivariant splitting} we see that $P$ is a direct summand of $\rho_{{R_\tau}}^n$.
\end{proof}

\begin{corollary}
\label{identification of projective C_2-equivariant Lambda-modules}
For the trivial action of the group $\mathbb{Z}/2\mathbb{Z}$ on $\spec \Lambda$ the category of $\mathbb{Z}/2\mathbb{Z}$-equivariant finitely generated projective $\Lambda$-modules is equivalent to the category of finitely generated projective $\Lambda[\mathbb{Z}/2\mathbb{Z}]$-modules. Thus, for a constructible and flat sheaf $\mathcal{F}$ of $\Lambda$-modules on $(\spec \mathbb{R})_{\acute{e}t}$, the stalk $\mathcal{F}_x$
is a finitely generated projective $\Lambda[\mathbb{Z}/2\mathbb{Z}]$-module.
\end{corollary}
\begin{proof}
Follows from Lemma \ref{Equivalent Category of Summands in Regular Representations} and discussion in the beginning of this subsection in view of the assumption that the characteristic of $\Lambda$ is different from 2.
\end{proof}

\subsection{Equivalence of  $D^b_{ctf}((\spec \mathbb{R})_{\acute{e}t},\Lambda)$ with 
$D^{b}(Proj(\Lambda[G]))$ and 
Constructible Witt groups of $\spec \mathbb{R}$} Using the description of stalks
of flat and constructible sheaves on $(\spec \mathbb{R})_{\acute{e}t}$  as finitely generated projective $\Lambda[\mathbb{Z}/2\mathbb{Z}]$-modules 
in Corollary \ref{identification of projective C_2-equivariant Lambda-modules} we now describe the constructible Witt theory of $\Lambda$-modules on $\spec \mathbb{R}$ as a $\mathbb{Z}/2\mathbb{Z}$-equivariant Witt theory of $\Lambda$.
The equivalence \[\phi: Shv((\spec \mathbb{R})_{\acute{e}t}, \Lambda)\rightarrow 
Mod(\Lambda[\mathbb{Z}/2\mathbb{Z}]) \]
 in Theorem \ref{Most useful corollary}
 is an exact equivalence of categories and induces an exact equivalence on the associated categories of chain complexes
\[\phi: Ch(Shv((\spec \mathbb{R})_{\acute{e}t}, \Lambda))\rightarrow 
Ch(Mod(\Lambda[\mathbb{Z}/2\mathbb{Z}])). \]
Let 
\begin{equation}
\label{Phi} 
\Phi: D((\spec \mathbb{R})_{\acute{e}t}, \Lambda)\rightarrow D(Mod(\Lambda[\mathbb{Z}/2\mathbb{Z}]))
\end{equation}
be the induced functor on the derived categories. 
The category $D^{b}_{ctf}((\spec \mathbb{R})_{\acute{e}t},\Lambda)$ of bounded complexes of sheaves of $\Lambda$-modules having Tor-finite dimension and constructible cohomology sheaves can now be identified with the more familiar bounded derived category of 
$D^b(Proj(\Lambda[\mathbb{Z}/2\mathbb{Z}]))$ of finitely generated projective $\Lambda[\mathbb{Z}/2\mathbb{Z}]$-modules giving us the the identification of the constructible Witt theory of $\spec \mathbb{R}$ with a Witt theory of finitely generated projective $\Lambda[\mathbb{Z}/2\mathbb{Z}]$-modules.

\begin{theorem}
\label{Main equivalence of derived categories}
Let $\Lambda$ a ring of finite characteristic not equal to $2$. Then the 
functor in $\mathrm{(\ref{Phi})}$
induces an equivalence of triangulated categories with duality 
$$\Phi: D^{b}_{ctf}((\spec \mathbb{R})_{\acute{e}t},\Lambda), \mathcal{H}om(-,\Lambda))\rightarrow (D^{b}(Proj(\Lambda[\mathbb{Z}/2\mathbb{Z}])), Hom_{\Lambda[\mathbb{Z}/2\mathbb{Z}]}(-, \Lambda))$$
and induces an isomorphism of constructible Witt theory of $\spec \mathbb{R}$
with an equivariant Witt theory of finitely generated projective $\Lambda$-modules for the action of the group $\mathbb{Z}/2\mathbb{Z}$ (the absolute Galois group of $\mathbb{R}$).
\end{theorem}
\begin{proof}
The category $D^b_{ctf}((\spec \mathbb{R})_{\acute{e}t},\Lambda)
\subset D((\spec \mathbb{R})_{\acute{e}t}, \Lambda)$ is the full triangulated subcategory consisting of complexes that are quasi-isomorphic in $D((\spec \mathbb{R})_{\acute{e}t}, \Lambda)$ to bounded complexes whose components are flat and constructible sheaves of $\Lambda$-modules. By Corollary \ref{identification of projective C_2-equivariant Lambda-modules} we have
the restriction
$$\Phi:D^{b}_{ctf}((\spec \mathbb{R})_{\acute{e}t},\Lambda) \rightarrow 
D^{b}(Proj(\Lambda[\mathbb{Z}/2\mathbb{Z}]))$$
of $\Phi$ is functor into the bounded derived category  $D^{b}(Proj(\Lambda[\mathbb{Z}/2\mathbb{Z}]))$ of finitely generated projective $\Lambda[\mathbb{Z}/2\mathbb{Z}]$-modules.
Recall that a  quasi-inverse 
\[\psi: Mod(\Lambda[\mathbb{Z}/2\mathbb{Z}])\rightarrow 
Shv((\spec \mathbb{R})_{\acute{e}t}, \Lambda)\]
of the functor $\phi: Shv((\spec \mathbb{R})_{\acute{e}t}, \Lambda)\rightarrow 
Mod(\Lambda[\mathbb{Z}/2\mathbb{Z}])$ 
is defined by associating to a $\Lambda[\mathbb{Z}/2\mathbb{Z}]$-module $M$ its $\mathbb{Z}/2\mathbb{Z}$-fixed points $M^{\mathbb{Z}/2\mathbb{Z}}$ to the base scheme $\spec \mathbb{R}$ and  the module $M$ to the scheme $\spec \mathbb{C}$. It is an exact functor and 
induces a quasi-inverse 
\[\Psi : D^{b}(Proj(\Lambda , \mathbb{Z}/2\mathbb{Z})) \rightarrow D^{b}_{ctf}((\spec \mathbb{R})_{\acute{e}t},\Lambda)\]
of $\Phi$. The diagram of triangulated categories 
\[
\xymatrix{
D_{ctf}^b((\spec \mathbb{R})_{\acute{e}t},\Lambda)^{op} \ar[rrr]^{\ \ \ \mathcal{H}om(-, \Lambda)\ \ \ } \ar[d]^{\Phi} &&& D_{ctf}^b((\spec \mathbb{R})_{\acute{e}t},\Lambda)\ar[d]^{\Phi}\\
D^{b}(Proj(\Lambda[\mathbb{Z}/2\mathbb{Z}]))^{op} \ar[rrr]^{ Hom_{\Lambda[\mathbb{Z}/2\mathbb{Z}]}(-, \Lambda) \ } &&&D^{b}(Proj(\Lambda[\mathbb{Z}/2\mathbb{Z}])) }
\]
commutes since the sheaf $\Lambda$ defining the duality $\mathcal{H}om(-, \Lambda)$ is the constant sheaf. Therefore the equivalence $\Phi$ is an equivalence of triangulated categories with duality. This completes the proof.
\end{proof}

\section{constructible Witt theory of complex algebraic varieties}
\label{section 6}
In this section, for a smooth complex algebraic variety $X$, we will identify the constructible Witt theory of \'etale sheaves of $\Lambda$-modules for finite commutative ring $\Lambda$ in which $2$ is a unit, with the constructible Witt theory of the underlying analytic space $X^{an}$ defined topologically in \cite{woolf2008witt}. 

\subsection{Algebraic and topological  comparison}
For a topological space $T$ and a ring $A$, an $A$-local system on $T$ is a sheaf of $A$-modules on $T$ that is locally constant. We denote by $Loc_A(T)$ the category of $A$-local systems on $T$. If $T$ is path-connected, then the category $Rep_A(\pi_1(T,p))$, consisting of representations of the fundamental group $\pi_1(T,p)$ on $A$-modules, is equivalent to the category $Loc_A(T)$ for any chosen point $p\in T$.

Let $X$ be a smooth complex algebraic variety, and $\Lambda$ a finite commutative ring. Denote by $X^{an}$ the underlying analytic space of complex points of the variety $X$. Then $X^{an}$ becomes a complex manifold with the classical Euclidean topology. Let $D^b_c(X^{an},\Lambda)$ denote the derived category of complexes of sheaves on the topological space $X^{an}$ with constructible cohomologies. Then we have an equivalence of categories $Loc_\Lambda(X_{\acute{e}t}) \simeq Rep_\Lambda(\pi_1(X^{an},x))$ given by $\mathcal{F}\mapsto (\mathcal{F}^{an})_x$ for a point $x\in X^{an}$. Here   $Loc_\Lambda(X_{\acute{e}t})$ denotes the category of locally constant sheaves of $\Lambda$-modules on $X_{\acute{e}t}$.

The following theorem, proved in \cite[Section 6]{MR751966}, is foundational to the comparison result of topological and algebraic derived categories of locally constant constructible sheaves.

\begin{theorem}
\label{BBD result}
For a smooth algebraic variety $X$ over $\mathbb{C}$, there exist a morphism of topoi
\[ \varepsilon: X^{an}\rightarrow X_{\acute{e}t}.  \]
Given a finite ring $\Lambda$, this induces an equivalence of categories:
\[
\varepsilon^*:
\left\{ \parbox{62mm}{\raggedright Category of constructible sheaves of $\Lambda$-modules on $X_{\acute{e}t}$}\right\}\xrightarrow{\simeq}
\left\{ \parbox{62mm}{\raggedright Category of constructible sheaves of $\Lambda$-modules on $X^{an} \color{white}{_a}$}\right\}
.\]
This equivalence further extends to derived categories, giving: 

$$\varepsilon^* : D^b_c(X_{\acute{e}t},\Lambda) \xrightarrow{\simeq} D^b_c(X^{an},\Lambda),$$
where $D^b_c(X_{\acute{e}t},\Lambda)$ denotes the bounded derived category of complexes of sheaves of $\Lambda$-modules on the $X_{\acute{e}t}$ with constructible cohomologies. Furthermore this equivalence is compatible with the six-functor formalism.

\end{theorem}

The equivalence in \ref{BBD result} is compatible with six-functor formalism and it induces an equivalence which is compatible with six functor formalism.
There is another identification of the category $D^b_c(X_{\acute{e}t},\Lambda)$. Recall that compact objects in a triangulated category are those objects $K$ for which $Hom(K,-)$ commutes with arbitrary direct sums. Following proposition gives that identification.

\begin{proposition}\label{Compact objects}
    Let $X$ be a smooth algebraic variety over $\mathbb{C}$ and $\Lambda$ be a finite ring. Then $D^b_c(X_{\acute{e}t},\Lambda)$ is exactly the full subcategory generated by compact objects in $D(X_{\acute{e}t},\Lambda)$.
\end{proposition}

\begin{proof}
See \cite[Proposition 6.4.8]{bhatt2013pro} for the proof.    
\end{proof}

Observe that the full subcategory generated by compact objects in $D(X_{\acute{e}t},\Lambda)$ is same as the category $D^b_{ctf}(X_{\acute{e}t},\Lambda)$ (see \cite[Section 4]{deligne1977seminaire}). Therefore Theorem \ref{BBD result} and Proposition \ref{Compact objects}  provide the following equivalence of derived categories: 
\begin{equation}\label{algebraic and topological}
D^b_{ctf}(X_{\acute{e}t},\Lambda) \xrightarrow{\simeq} D^b_c(X^{an},\Lambda)
\end{equation}
which is compatible with the six-functor formalism.

\begin{theorem}\label{Top-Alg}
Let $X$ be a smooth algebraic variety over $\mathbb{C}$ and $\Lambda$ be a finite ring with characteristic not equal to $2$. Then equivalence in $\mathrm{(\ref{algebraic and topological})}$ is an equivalence of triangulated categories with duality
$$(D^b_{ctf}(X_{\acute{e}t},\Lambda), R\mathcal{H}om(-,\Lambda) \xrightarrow{\simeq} (D^b_c(X^{an},\Lambda),R\mathcal{H}om(-,\Lambda)).$$
Thus, for a finite ring $\Lambda$ of characteristic not equal to $2$, we have an isomorphism
\[W^i_c(X_{\acute{e}t}, \Lambda)\xrightarrow{\ \sim \ } W^i_c(X^{an}, \Lambda)\]
of the constructible Witt theory of the scheme $X$ with the constructible Witt theory of the topological space $X^{an}$ defined in 
\cite{woolf2008witt} - here we are using the notation  
$W^i_c(X^{an}, \Lambda)$ instead of $W^c_i(X^{an}, \Lambda)$ used in \cite{woolf2008witt}.
\end{theorem}
\begin{proof} The theorem follows by the result that 
the equivalence in (\ref{algebraic and topological}) is compatible with the six-functor formalism and is an equivalence between triangulated categories with dualities. See Theorem \ref{BBD result} and Proposition \ref{Compact objects}.
\end{proof}

\section{Witt-valued signatures for real and complex projective varieties}
\label{signature}
In \cite{schurmann2020witt} Sch\"{u}rmann and Woolf have provided signature-type invariants for topological spaces taking values in Witt theories related to the coefficient ring.
For projective real and complex algebraic varieties now we can describe algebraic analogs of signature considered by Sch\"{u}rmann and Woolf via the constructible Witt theory developed in this paper. 
 
\subsection{Signatures for projective real and complex algebraic varieties}

For a Noetherian ring $\Lambda$ of positive characteristic such that $2$ is invertible in $\Lambda$,  the Theorem \ref{Main equivalence of derived categories} provides an equivalence 
$$ D^{b}_{ctf}((\spec \mathbb{R})_{\acute{e}t}, \Lambda) \xrightarrow{\simeq} D^{b}(Proj(\Lambda [\mathbb{Z}/2\mathbb{Z}]))$$
of the derived category of bounded complexes sheaves of $\Lambda$-modules having Tor-finite dimension and constructible cohomology sheaves with the bounded derived category of finitely generated projective $\Lambda[\mathbb{Z}/2\mathbb{Z}]$-modules. 
We also have an equivalence of triangulated categories with duality
\[D^{b}_{ctf}((\spec \mathbb{C})_{\acute{e}t}, \Lambda) \xrightarrow{\simeq} D^{b}(Proj(\Lambda)) \]
In Theorem \ref{algebraic analog of signature} these equivalences allow us to identify the constructible Witt theory of $\spec \mathbb{R}$ with a  $\mathbb{Z}/2\mathbb{Z}$-equivariant Witt theory of finitely generated projective $\Lambda$-modules; and, the constructible Witt theory of 
$\spec \mathbb{C}$ with the Witt theory of finitely generated projective $\Lambda$-modules. The theorem below gives the algebraic analog of signature defined in \cite{woolf2008witt}.

\begin{theorem}\label{algebraic analog of signature} 
For a ring $R$ denote the  Witt groups of finitely generated projective $R$-modules by $W^i_{lf}(R)$.
$(a.)$ Let $X$ be a projective real algebraic variety with the structure morphism $f:X\rightarrow \spec \mathbb{R}$ and $\Lambda$ a ring of finite characteristic not equal to $2$. Then we have the induced proper pushforward for constructible Witt theory 
\[W^{i}(f_*): W^i_c(X_{\acute{e}t}, \Lambda)\rightarrow W^i_c((\spec \mathbb{R})_{\acute{e}t}, \Lambda)=W^i_{lf}(\Lambda[\mathbb{Z}/2\mathbb{Z}]).\] 
(b.) Let $X$ be a projective complex algebraic variety with the structure map $f:X\rightarrow \spec \mathbb{C}$, and a finite commutative ring $\Lambda$ of characteristic not equal to $2$. Then we have the induced pushforward 
\[W^{i}(f_*): W^i_c(X_{\acute{e}t}, \Lambda)
%=W^i_c(X^{an}, \Lambda)
\rightarrow W^i_c((\spec \mathbb{C})_{\acute{e}t}, \Lambda)=W^i_{lf}(\Lambda)\] 
for constructible Witt theory of $X$.
\end{theorem}

\begin{proof}
The proof follows from Proposition \ref{derived category pullbacks and pushforwards}, and Theorems \ref{Main equivalence of derived categories} and \ref{Top-Alg}.
\end{proof}

\subsection{Final Remarks} The study of constructible Witt groups in \cite{woolf2008witt} and \cite{schurmann2020witt} is inspired by cobordism theory of Witt spaces by Siegel in \cite{Siegel1983WITTSA}. For a smooth complex projective variety $X$, the analytic space $X^{an}$ is a compact Witt space and it is proved in \cite[Corollary 5.12]{woolf2008witt} that the natural maps
\[ \Omega_i^{Witt}(X^{an}) \rightarrow W^i_c(X^{an}, \mathbb{Q})   \]
are isomorphisms for $i> \text{dim}X$: Here $\Omega_i^{Witt}(X^{an})$ denotes the Witt bordism group of $i$-dimensional Witt spaces over $X^{an}$. With this identification in mind, the most interesting question for us is the need of the corresponding algebraic theory of bordism of Witt spaces in algebraic geometry
and understanding the analog of the above isomorphism algebraically, and in the case of complex algebraic varieties - its compatibility with the isomorphism in Theorem \ref{Top-Alg}. Such a theory will also hopefully provide interesting interpretation of the signature in Theorem \ref{algebraic analog of signature} in case of real algebraic varieties.
Authors are exploring these based on the work by Levine and Pandharipande in \cite{LevinePandharipande2009} on geometrically constructed algebraic cobordism.

\bibliography{references}

\end{document}